\documentclass[a4paper, 11pt, reqno]{amsart}
\usepackage[usenames,dvipsnames]{color}
\usepackage{amssymb, amsmath, latexsym, graphics, graphicx}
\newtheorem{theorem}{Theorem}[section]
\newtheorem{lemma}[theorem]{Lemma}
\newtheorem{corollary}[theorem]{Corollary}

\newtheorem{proposition}[theorem]{Proposition}

\makeatother

\newcommand{\ul}{\underline}

\begin{document}
\title{Identities in unitriangular and gossip monoids}
\maketitle

\begin{center}
MARIANNE JOHNSON\footnote{School of Mathematics, University of Manchester,
Manchester M13 9PL, UK.\\
Email \texttt{Marianne.Johnson@maths.manchester.ac.uk}}
AND PETER FENNER\footnote{School of Mathematics, University of Manchester,
Manchester M13 9PL, UK.\\
Email \texttt{Peter.Fenner@postgrad.manchester.ac.uk}.\\
Peter Fenner's research is supported by an EPSRC Doctoral Training Award.}

\date{\today}
\thanks{}
\end{center}

\begin{abstract}
We establish a criterion for a semigroup identity to hold in the monoid of  $n \times n$ upper unitriangular matrices with entries in a commutative semiring $S$. This criterion is combinatorial modulo the arithmetic of the multiplicative identity element of $S$. In the case where $S$ is idempotent, the generated variety is the variety $\mathbf{J_{n-1}}$, which by a result of Volkov is generated by any one of: the monoid of unitriangular Boolean matrices, the monoid $R_n$ of all reflexive relations on an $n$ element set, or the Catalan monoid $C_n$. We propose $S$-matrix analogues of these latter two monoids in the case where $S$ is an idempotent semiring whose multiplicative identity element is the `top' element with respect to the natural partial order on $S$, and show that each generates $\mathbf{J_{n-1}}$. As a consequence we obtain a complete solution to the finite basis problem for lossy gossip monoids.
\end{abstract}

\section{Introduction}
\label{sec_intro}
The \emph{finite basis problem} for semigroups asks: which semigroups have an equational theory admitting a finite basis of identities? Such semigroups are called finitely based. In contrast to the situation for finite groups \cite{OP64}, it has long been known that there exist \emph{finite} semigroups which are \emph{non-finitely based} \cite{P69}, and there is a rich literature studying the finite basis problem from viewpoint of finite semigroups (see the survey \cite{V01}). As observed by Volkov \cite{V15}, infinite semigroups are far less frequently studied in the context of the finite basis problem, due to the fact that many natural infinite semigroups are in some sense `too big' to allow for the kind of universal coincidences demanded by identities. For example, if $S$ is a commutative semiring into which the semiring of natural numbers can be embedded, then for $n>1$ the monoid of all $n\times n$ (upper triangular) matrices over $S$ satisfies no non-trivial identities, since the free monoid of rank $2$ embeds into all such semigroups (see \cite{V15} for example). The finite basis problem is increasingly studied for families of infinite semigroups of combinatorial interest for which identities are known to exist, with complete results available for one-relator semigroups \cite{Shn89} and Kauffman monoids \cite{ACHLV15}, and several recent partial results for various semigroups of upper triangular matrices with restrictions on the size of the matrices and the entries permitted on the diagonals \cite{CHL16, CHLS16,V15, ZJL17}.

In this work we consider the identities satisfied by several families of matrix semigroups, beginning with upper triangular matrices with entries in a commutative semiring. In Section \ref{sec_triangular} we show how the analysis of \cite{DJK17} may be generalised to the setting of commutative semirings $S$ to provide necessary and sufficient conditions for a semigroup identity to hold in the monoid of  $n \times n$ upper triangular matrices with entries in $S$. This result is then applied in Section \ref{sec_unitriangular} to establish a criterion for a semigroup identity to hold in the submonoid of  $n \times n$ upper unitriangular matrices, showing that the generated variety depends only upon the isomorphism type of the subsemiring generated by the multipllicative identity element of $S$. In the case where $S$ is an \emph{idempotent semiring} our result together with a result of Volkov \cite{V04} yields that the generated variety is $\mathbf{J_{n-1}}$, that is, the variety of semigroups generated by the monoids of height $n-1$ in Simon's hierarchy of finite $\mathcal{J}$-trivial monoids \cite{Simon}.  In Sections \ref{sec_reflexive} we introduce the submonoid $R_n(S)$ of the full matrix monoid over an \emph{interval semiring} $S$, and show that this generates the same variety as its finite (Boolean) counterpart, the reflexive monoid. In Section \ref{sec_catalan} we consider several monoids related to the Catalan monoid, including the so-called \emph{lossy gossip monoid} $\mathcal{G}_n$ (that is, the monoid generated by all ``metric'' matrices in the full  matrix monoid over the tropical semiring \cite{BDF15}). By \cite{V04} this common variety is once again seen to be $\mathbf{J_{n-1}}$. Since the variety $\mathbf{J_{n-1}}$ is known to be finitely based for $n \leq 4$, and non-finitely based otherwise, this settles the finite basis problem for the above mentioned families of  monoids.

We conclude this introduction by briefly recalling the necessary definitions, notation and background.

\subsection*{Semigroup identities}
\label{subsec_identities}
We write $\mathbb{N}_0$ and $\mathbb{N}$ respectively for the natural numbers with and without $0$. If $\Sigma$ is a finite alphabet, then $\Sigma^+$ will denote the free semigroup on $\Sigma$, that is, the set of
finite, non-empty words over $\Sigma$ under the operation of concatenation. Likewise, $\Sigma^*$ will denote the free monoid on $\Sigma$. Thus $\Sigma^* = \Sigma^+ \cup \{1\}$ where $1$ denotes the empty word. For $w \in \Sigma^+$ and $s \in \Sigma$ we write $|w|$ for
the length of $w$ and $|w|_s$ for the number of occurrences of the letter $s$ in $w$. For $1 \leq i \leq |w|$ we write $w_i$ to denote the $i$th letter of $w$. The \textit{content} of $w$ is the map
$\Sigma \to \mathbb{N}_0, s \mapsto |w|_s$.

Recall that a (\textit{semigroup}) \textit{identity} is a pair of words, usually written ``$u=v$'', in the free semigroup $\Sigma^+$ on an alphabet $\Sigma$. The identity is said to be \emph{balanced} if $|u|_a=|v|_a$ for all $a \in \Sigma$. We say that the identity \textit{holds} in a semigroup $U$ (or that $U$ \textit{satisfies} the identity) if every morphism from $\Sigma^+$ to $U$ maps $u$ and $v$ to the same element of $U$. If a morphism maps $u$ and $v$ to the same element we say that it \textit{satisfies} the given identity in $U$; otherwise it \textit{falsifies} it. We write ${\rm Id}(U)$ to denote the set of all identities satisfied by the semigroup $U$.

\subsection*{Semirings}
\label{subsec_semiring}
Let $S$ be a commutative semiring with additive identity $0_S$ and multiplicative identity $1_S$. We say that $S$ is \emph{idempotent} if $a + a = a$ for all $a \in S$. Examples include the Boolean semiring $\mathbb{B} =\{0,1\}$ in which the only undetermined operation is defined by $1+1=1$, and the tropical semifield $(\mathbb{R}\cup\{-\infty\}, \oplus, \otimes)$, where $a\oplus b = {\rm max}(a,b)$ and $a \otimes b = a+b$, and in which $-\infty$ is the `zero' element, and $0$ is the `one'. There is a natural partial order on every idempotent semiring $S$ given by $a \leq b$ if and only if $a+b=b$; it is clear from definition that $a+b \geq a, b$ for all $a,b \in S$. Thus $0_S$ is the least element of $S$ with respect to this order. Moreover, if $a \leq b$ in $S$, then $cad \leq cbd$ and $a+c \leq b+c$ for all $c,d \in S$.

We say that a commutative semiring $S$ is an \emph{interval semiring} if $S$ is idempotent, and $1_S$ is the greatest element of $S$ with respect to the natural partial order on $S$.  Examples of interval semirings include: the Boolean semiring $\mathbb{B}$; the semiring $\mathcal{I} =([0,1],\cdot, \oplus)$ with usual multiplication of numbers and addition given by taking the maximum;  the semiring $(\mathbb{R}_{\leq 0} \cup \{-\infty\}, \otimes, \oplus)$ with multiplication given by usual addition of numbers and addition given by taking the maximum; any distributive lattice $L$ with addition $\vee$ and multiplication $\wedge$.

\subsection*{Matrix semigroups}
 It is easy to see that the set of all $n \times n$ matrices with entries in $S$ forms a monoid under the matrix multiplication induced from the operations in $S$. We denote this semigroup by $M_n(S)$ and write $UT_n(S)$ to denote the subsemigroup of $M_n(S)$ consisting of the upper-triangular matrices in $M_n(S)$ whose entries below the main diagonal are zero. We also write $U_n(S)$ to denote the semigroup of unitriangular matrices, namely those elements of $UT_n(S)$ whose diagonal entries are all equal to $1_S$.

In the case where $S$ is an idempotent semiring we define a partial order $\preceq$ on $M_n(S)$ by $A \preceq B$ if and only if $A_{i,j} \leq B_{i,j}$ for all $i$ and $j$. It is easy to see that matrix multiplication respects the partial order $\preceq$ (i.e. $M_n(S)$ is an ordered monoid). Indeed, for $A, B,C \in M_n(S)$ with $A \preceq B$, for all $i,j$ we have
$$(CA)_{i,j} = \sum_{k=1}^n C_{i,k}A_{k,j} \leq \sum_{k=1}^n C_{i,k}B_{k,j} = (CB)_{i,j}.$$

\subsection*{Polynomials}
\label{subsec_poly}
By a \textit{formal polynomial} in variables from a set $X$ we mean an element of the commutative polynomial semiring $S[X]$, that is, a finite formal sum in which each term is a formal product of a non-zero coefficient from $S$ and formal powers of finitely many of the variables of $X$, considered up to the commutative and distributive laws in $S$. If $S$ is idempotent, we consider the summation up to idempotency of addition.

Each formal polynomial naturally defines a function from $S^X$ to $S$, by interpreting all formal products and formal sums as products and sums within $S$. Two distinct formal polynomials may define the same function. For example, $x^{\otimes 2} \oplus x \oplus 1$ and $x^{\otimes 2} \oplus 1$ are distinct formal tropical polynomials defining the same function, since $x$ can never exceed both $x^{\otimes 2}$ and $1$. We say that two formal polynomials are \textit{functionally equivalent} over $S$ if they represent the same function from $S^X$ to $S$.

\section{The identities of triangular matrices}
\label{sec_triangular}
We begin by providing a generalisation of \cite[Theorem 5.1]{DJK17} to the setting of commutative semirings.

Let $[n] = \{1, 2,\ldots, n\}$. By a \textit{$k$-vertex walk} (or \textit{walk of vertex length} $k$) in $[n]$ we mean a $k$-tuple $(v_1, \dots, v_k)$ such that $v_1 \leq v_2 \leq \dots \leq v_k$. A \textit{$k$-vertex path} (or \textit{path of vertex length $k$}) is a $k$-vertex walk in which consecutive vertices (and hence all vertices) are distinct.

Let $w$ be a word over the alphabet $\Sigma$. For $0 \leq p < q \leq |w|+1$ and $s \in \Sigma$ we
define
$$\beta_s^w(p, q) \ = \ | \lbrace i \in \mathbb{N} \mid p < i < q, w_i = s \rbrace |$$
to be the number of occurrences of $s$ lying strictly \emph{between} $w_{p}$ and $w_{q}$. For each $u \in \Sigma^*$ with $|u| \leq n-1$ and each $(|u|+1)$-vertex path $\rho=(\rho_0, \rho_1, \ldots, \rho_{|u|})$ in $[n]$, we define a formal polynomial (over an arbitrary, but fixed, commutative semiring $S$) having variables $x(s,i)$ for each letter $s \in \Sigma$ and vertex $i \in [n]$ as follows:
\begin{eqnarray*}
f_{u, \rho}^{w} = \sum \prod_{s\in \Sigma}\prod_{k=0}^{|u|} x(s, \rho_k)^{\beta_s^w(\alpha_k, \alpha_{k+1})},
\end{eqnarray*}
where the sum ranges over all $0=\alpha_0< \alpha_1<\cdots<\alpha_{|u|}<\alpha_{|u|+1}=|w|+1$ such that $w_{\alpha_k}=u_k$ for $k=1, \ldots, |u|$. Thus it is easy to see that $f_{u, \rho}^{w} \neq 0_S$ if and only if $u$ is a scattered subword of $w$ of length equal to the path $\rho$. Note that taking $u$ to be the empty word forces $\rho = (\rho_0)$ for some $\rho_0 \in [n]$ and hence $f_{u, \rho}^w = \prod_{s \in \Sigma} x(s, \rho_0)^{|w|_s}$ is a monomial completely determined by the content of $w$.

\begin{lemma}\label{entries}
Let $S$ be a commutative semiring, and let $\phi : \Sigma^+ \to UT_n(S)$ be a morphism. Define $\underline{x} \in S^{\Sigma \times [n]}$ by
$$\underline{x}(s,i) \ = \ \phi(s)_{i,i}.$$
Then for any word $w \in \Sigma^+$ and vertices $i, j \in [n]$ we have
\begin{equation}\label{eq1}
\phi(w)_{i,j} \ = \ \sum_{\substack{u \in \Sigma^*,\\ |u| \leq n-1}}\sum_{\rho \in [n]_{i,j}^{|u|}} \left(\prod_{k=1}^{|u|} \phi(u_k)_{\rho_{k-1}, \rho_k}\right) \cdot f_{u, \rho}^w(\underline{x}),
\end{equation}
where $[n]^{|u|}_{i,j}$ denotes the set of all $(|u|+1)$-vertex paths from $i$ to $j$ in $[n]$.
\end{lemma}

\begin{proof}

We follow the proof given in \cite{DJK17}.

Let $i$ and $j$ be vertices. Using the definition of the functions $f_{u,\rho}^w$, the value given to $\underline{x}$ and the distributivity of multiplication over addition, the right-hand-side of \eqref{eq1} is equal to
\begin{align*}
\sum_{\substack{u \in \Sigma^*,\\ |u| \leq n-1}} \sum_{\alpha \in \mathcal{A}^w_{u}}
\sum_{\rho \in [n]_{i,j}^{|u|}} \left(\prod_{k=1}^{|u|} \phi(u_k)_{\rho_{k-1}, \rho_k}\right) \cdot
\left( \prod_{s\in \Sigma}\prod_{k=0}^{|u|} (\phi(s)_{\rho_k, \rho_k})^{\beta_s^w(\alpha_k, \alpha_{k+1})} \right)
\end{align*}
where
$$\mathcal{A}^w_{u} =\{(\alpha_0, \ldots, \alpha_{|u|+1}): 0=\alpha_0< \alpha_1<\cdots<\alpha_{|u|}<\alpha_{|u|+1}=|w|+1 \mbox{ with } w_{\alpha_k}=u_k\}.$$
Notice that we are summing over all possible words $u$ of length less than $n$, and then over all scattered subwords of $w$ equal to $u$. Thus, we are simply summing over all scattered subwords of $w$ of length less than $n$, so the above is equal to:
\begin{align*}
\sum_{l=0}^{n-1}\sum_{\alpha \in \mathcal{A}_{l}}\sum_{\rho \in [n]_{i,j}^l} \left(\prod_{k=1}^{l} \phi(w_{\alpha_k})_{\rho_{k-1}, \rho_k}\right) \cdot
\left( \prod_{s\in \Sigma}\cdot \prod_{k=0}^{l} (\phi(s)_{\rho_k, \rho_k})^{\beta_s^w(\alpha_k, \alpha_{k+1})} \right)
\end{align*}
where $\mathcal{A}_{l} =\{(\alpha_0, \ldots, \alpha_{l+1}): 0=\alpha_0< \alpha_1<\cdots<\alpha_{l}<\alpha_{l+1}=|w|+1\}$.

Now to each term in the above sum, defined by a choice of $\alpha_i$'s and a $\rho \in [n]_{i,j}^l$, we can associate
a $(|w|+1)$-vertex walk $(\sigma_0 = i, \dots, \sigma_{|w|} = j)$ in $[n]$ whose underlying path is $\rho$ and which transitions to vertex $\rho_k$ after $\alpha_{k}$ steps. Clearly every $(|w|+1)$-vertex walk from $i$ to $j$ arises exactly once in this way, and so we are summing over all such walks.
In each term, the first bracket gives a factor $\phi(w_q)_{\sigma_{q-1},\sigma_{q}}$ when $q = \alpha_k$ for some $k$,
while from the definition of the functions $\beta_s^w$, the second bracket gives a factor $\phi(w_q)_{\sigma_{q-1},\sigma_q}$ for each $q$ not of this form. Thus, the above is simply equal to:
$$\sum \prod_{q=1}^{|w|} \phi(w_q)_{\sigma_{q-1},\sigma_q}$$
where the sum is taken over all $(|w|+1)$-vertex walks $(i = \sigma_0, \sigma_1, \dots, \sigma_{|w|} = j)$ in $[n]$. But by the
definition of multiplication in $UT_n(S)$, this is easily seen to be equal to $\left( \phi(w_1)\cdots \phi(w_{|w|}) \right)_{i,j} \ = \ \phi(w)_{i,j}.$
\end{proof}

Let  $f_u^w$ denote the polynomial $f_{u, \rho}^w$ with $\rho=(1,2,\ldots, |u|+1)$ in variables $x(s,i)$, with $s \in \Sigma$ and $1 \leq i \leq |u|+1$. We are now ready to prove the main theorem of this section.

\begin{theorem}
\label{thm_kletterGamma}
Let $S$ be a commutative semiring. The identity $w=v$ over alphabet $\Sigma$ is satisfied in $UT_n(S)$ if and only if for every $u \in \Sigma^*$ with $|u| \leq n-1$ the polynomials $f_{u}^w$ and $f_{u}^v$ are functionally equivalent over $S$.
\end{theorem}

\begin{proof}
Suppose first that $f_{u}^w(\underline{x})\neq f_{u}^v(\underline{x})$ for some word $u \in \Sigma^{+}$ of length at most $n-1$ and $\underline{x} \in S^{\Sigma \times [n]}$. Define a morphism $\phi : \Sigma^+ \to UT_n(S)$ by
$$\phi(s)_{p,p} \ = \ \underline{x}(s,p) \in S, \mbox{ for all } p \in [n] \mbox{ and } s \in \Sigma; \mbox{ and}$$
$$\phi(s)_{p,q} \ = \ \begin{cases}
1_S & \mbox{if } s=u_i, p=i, q=i+1\\
0_S & \mbox{otherwise}.
\end{cases}$$
Then by Lemma~\ref{entries},
$$\phi(v)_{i,j} \ = \ f_{u}^v(\ul{x}) \ \neq \ f_{u}^w(\ul{x}) \ = \ \phi(w)_{i,j},$$
and so the morphism $\phi$ falsifies the identity in $UT_n(S)$.

Conversely, suppose that $f_{u}^w$ and $f_{u}^v$ are functionally equivalent over $S$ for all $u\in \Sigma^*$ of length at most $n-1$. Noting that for any path $\rho$ of vertex length $|u|$, the polynomials $f_{u, \rho}^w$ and $f_{u}^w$ differ only in the labelling of their variables, it is then easy to see that $f_{u, \rho}^w$ and $f_{u, \rho}^v$ are functionally equivalent for all pairs $u, \rho$ with $u\in \Sigma^*$ of length at most $n-1$, and $\rho$ a path of vertex length $|u|+1$ through $[n]$.

It suffices to show that the identity $w=v$ is satisfied by every morphism $\phi : \Sigma^+ \to UT_n(S)$, so let $\phi$ be such a morphism and define $\underline{x} \in S^{\Sigma \times [n]}$ by $\underline{x}(s,i) = \phi(s)_{i,i}$. Since $\phi$ is a morphism to $UT_n(S)$, we know that $\phi(w)_{i,j}=0_S = \phi(v)_{i,j}$ whenever $i > j$. On the other hand, if $i \leq j$ then Lemma~\ref{entries} gives
$$\phi(w)_{i,j} \ = \ \sum_{\substack{u \in \Sigma^*,\\ |u| \leq n-1}} \sum_{\rho \in [n]_{i,j}^{|u|}} \left(\prod_{k=1}^{|u|} \phi(u_k)_{\rho_{k-1}, \rho_k}\right) \cdot f_{u, \rho}^w(\underline{x}) \ = \ \phi(v)_{i,j}.$$
\end{proof}

\begin{lemma}
\label{count}
Let $S$ be a semiring whose multiplicative monoid contains an element $\alpha$ generating a free submonoid of rank 1, and let $w, v \in \Sigma^+$.
\begin{itemize}
\item[(i)] The polynomials $f_1^w$ and $f_1^v$ are functionally equivalent if and only if $w$ and $v$ have the same content.
\item[(ii)]  Suppose further that the partial sums $\sum_{i=0}^{j}\alpha^i$  for $j \in \mathbb{N}_0$ are pairwise distinct. If  $f_a^w$ is functionally equivalent to $f_a^v$ for all $a \in \Sigma$, then $f_1^w$ is functionally equivalent to $f_1^v$ .
\end{itemize}
\end{lemma}

\begin{proof}
(i) By definition, $f_1^w$ is the monomial $\prod_{s \in \Sigma} x(s, 1)^{|w|_s}$, and it is clear that $w$ and $v$ have the same content if and only if the formal polynomials $f_1^w$ and $f_1^v$ are \emph{identical}. In particular, if the content of the two words agree, then these polynomials are functionally equivalent. Suppose then that $f_1^w$ and $f_1^v$ are functionally equivalent.  Setting $x(s,1)=\alpha$ and $x(t,1)=1_S$ for all $t \neq s$ then yields $\alpha^{|w|_s} = \alpha^{|v|_s}$, and hence $|w|_s=|v|_s$. Repeating this argument for each $s \in \Sigma$ yields that the two words have the same content.

(ii) It suffices to show that if  $f_a^w$ is functionally equivalent to $f_a^v$ for all $a \in \Sigma$, then the content of
the two words must be equal. Evaluating the polynomials $f_a^w$ and $f_a^v$ at  
$x(a, 1) = \alpha$ and $x(z, i) = 1_S$ for all other choices of $z, i$ yields $\sum_{i=0}^{|w|_a-1}\alpha^i$ = $\sum_{i=0}^{|v|_a-1}\alpha^i$, and hence $|w|_a=|v|_a$. Repeating this argument for each $a \in \Sigma$ gives that the two words have the same content.
\end{proof}

Under the hypothesis of the previous lemma, we note that identities must be balanced.
\begin{corollary}
\label{words}
Let $S$ be a semiring whose multiplicative monoid contains a free submonoid of rank $1$. Any non-trivial semigroup identity satisfied by $UT_n(S)$ or $M_n(S)$ must be balanced.
\end{corollary}
Under the stronger hypothesese of the Lemma \ref{count} (ii), we may reduce the number of polynomials to be checked by $1$.
\begin{corollary}
\label{words}
Let $S$ be a semiring whose multiplicative monoid contains an element $\alpha$ generating a free submonoid of rank 1, and suppose that the partial sums $\sum_{i=1}^{j}\alpha^i$  for $j \in \mathbb{N}$ are pairwise distinct. The identity $w=v$ over alphabet $\Sigma$ is satisfied in $UT_n(S)$ if and only if for every $u \in \Sigma^+$ with $1 \leq |u| \leq n-1$ the polynomials $f_{u}^w$ and $f_{u}^v$ are functionally equivalent.
\end{corollary}

\section{The identities of unitriangular matrices}
\label{sec_unitriangular}
Say that the \emph{scattered multiplicity} of $u\in \Sigma^+$ in $w \in \Sigma^+$ is the number of distinct ways in which $u$ occurs as a scattered subword of $w$, and denote this by $m_u^w \in \mathbb{N}_0$.
For $m \in \mathbb{N}_0$ write $\lfloor m \rfloor_S:=\sum_{j=1}^{m} 1_S$.
\begin{theorem}
\label{modS}
Let $S$ be a commutative semiring. The identity $w=v$ over alphabet $\Sigma$ is satisfied in $U_n(S)$ if and only if  $\lfloor m_u^w\rfloor_S=\lfloor m_u^v\rfloor_S$ for each word $u \in \Sigma^+$ of length at most $n-1$.
\end{theorem}

\begin{proof}
Let $\phi: \Sigma^+ \rightarrow U_n(S)$ be a morphism. Since every element of the image of $\phi$ has all diagonal entries equal to $1_S$ it follows from Lemma \ref{entries} and the definition of the polynomials $f_{u, \rho}^w$ that for all $1 \leq i< j \leq n$, we have
$$\phi(w)_{i,j} = \sum_{\substack{u \in \Sigma^+,\\ |u| \leq n-1}}\sum_{\rho \in [n]_{i,j}^{|u|}} \left(\prod_{k=1}^{|u|} \phi(u_k)_{\rho_{k-1}, \rho_k}\right) \cdot\lfloor m_u^w \rfloor_S,$$
where $m_u^w$ denotes the scattered multiplicity of $u$ in $w$. Since these multiplicities account for the only part of the formula which directly depends upon $w$, it is then clear that if each of the equalities $\lfloor m_u^w \rfloor_S= \lfloor m_u^v \rfloor_S$ holds, then we must have $w=v$ in $U_n(S)$.

Now suppose $w=v$ is satsified in $U_n(S)$ and let $u$ be a word of length $l <n$ with scattered multiplicities $m_u^w$ and $m_u^v$ in $w$ and $v$ respectively. Consider the morphism $\phi : \Sigma^+ \to U_n(S)$ defined by
$$\phi(s)_{p,p} \ = \ 1_S, \mbox{ for all } p \in [n] \mbox{ and } s \in \Sigma; \mbox{ and}$$
$$\phi(s)_{p,q} \ = \ \begin{cases}
1_S & \mbox{if } s=u_i, p=i, q=i+1\\
0_S & \mbox{otherwise}.
\end{cases}$$
Notice that  Lemma \ref{entries} then yields $\lfloor m_u^w \rfloor_S = \phi(w)_{1, l+1} = \phi(v)_{1, l+1}= \lfloor m_u^v\rfloor_S$.
\end{proof}

\begin{proposition}
Let $S$ and $T$ be commutative semirings. The semigroups $U_n(S)$ and $U_n(T)$ generate the same variety of semigroups if and only if $1_S$ and $1_T$ generate isomorphic semirings.
\end{proposition}

\begin{proof}
If $1_S$ and $1_T$ generate isomorphic semirings, then for all $j, k \in \mathbb{N}_0$ we have $\lfloor j \rfloor_S = \lfloor k \rfloor_S$ if and only if $\lfloor j \rfloor_T = \lfloor k \rfloor_T$. It then follows immediately from Theorem \ref{modS} that $U_n(S)$ and $U_n(T)$ satisfy exactly the same semigroup identities.

Conversely, if $U_n(S)$ and $U_n(T)$ satisfy the same identities, it follows that for all words $w,v,u \in \Sigma^+$ we must have $\lfloor m_u^w \rfloor_S= \lfloor m_u^v \rfloor_S$ if and only if $\lfloor m_u^w \rfloor_T= \lfloor m_u^v \rfloor_T$. Consideration of all pairs of words $w=a^j,v=a^k$ with respect to the fixed word $u=a$ of length $1$ allows us to determine all relations of the form $\lfloor j \rfloor_R = \lfloor k \rfloor_R$ for $j,k \in \mathbb{N}$ and $R=S, T$. Since the same set of relations holds for $R=S$ and $R=T$, it follows that $1_S$ and $1_T$ generate isomorphic semirings.
\end{proof}

\begin{corollary}
Let $S$ be an idempotent semiring. The identity $w=v$ over alphabet $\Sigma$ is satisfied in $U_n(S)$ if and only if $w$ and $v$ admit the same set of scattered subwords of length at most $n-1$.
\end{corollary}

\begin{proof}
If $S$ is idempotent then it is easy to see that 
$$\lfloor m_u^v\rfloor_S = \begin{cases}
1_S & \text{if  } u \text{ is a scattered subword of } w \\
0_S & \text{otherwise.} \\
\end{cases}$$
\end{proof}

The previous results generalise a result of Volkov \cite{V04}, who proved that $w=v$ is a semigroup identity for $U_n(\mathbb{B})$ if and only if $w$ and $v$ have the same scattered subwords of length at most $n-1$. Since the results of that paper also show that the  unitriangular Boolean matrices $U_n(\mathbb{B})$, the monoid $R_n$ of reflexive binary relations on a set of cardinality $n$, and the Catalan monoid $C_n$ all satisfy exactly the same set of identities, we get the following immediate corollary.

\begin{corollary}
Let $S$ be an idempotent semiring. The semigroup $U_n(S)$ satisfies exactly the same semigroup identities as the semigroup of reflexive relations $R_n$ or the Catalan monoid $C_n$.
\end{corollary}

The monoid $U_n(S)$ can be viewed as an oversemigroup of $U_n(\mathbb{B})$ allowing for entries over the idempotent semiring $S$, and so it is natural to ask if there are analogous extensions of $R_n$ and $C_n$. We note that there is an obvious Boolean matrix representation of $R_n$, formed by sending a relation $R$ to the Boolean matrix whose $i,j$th entry is $1$ if and only if $i$ and $j$ are related by $R$. In the following section we shall consider a natural analogue of $R_n$ consisting of matrices over a semiring $S$ with diagonal entries all equal to the multiplicative identity of $S$. It is clear that, in general, the set of all such matrices need \emph{not} form a semigroup (e.g. over the tropical semiring such matrices are not closed under multiplication). We shall therefore restrict our attention to a particular class of idempotent semirings.
\section{Generalised reflexive monoids}
\label{sec_reflexive}

\begin{lemma}
\label{Jtriv}
Let $S$ be an idempotent semiring, and let $V$ be a subsemigroup of $M_n(S)$ with the property that every element of $V$ has all diagonal entries equal to $1_S$.\\
\begin{itemize}
\item[(i)] If $A = U(1)X(1)U(2) \cdots U(L)X(L)U(L+1)$ and $B = X(1) \cdots X(L)$ for some $U(i), X(i) \in V$, then $B \preceq A$.
\item[(ii)] For all $A \in V$ we have
$$I_n \preceq A \preceq A^2 \preceq A^2 \preceq \cdots \preceq A^n \preceq \cdots$$
where $I_n$ denotes the identity matrix of $M_n(S)$.\\
(In particular, $V$ is $\mathcal{J}$-trivial and so every regular element of $V$ is idempotent.)
\end{itemize}
\end{lemma}

\begin{proof}
(i) Suppose that $A = U(1)X(1)U(2) \cdots U(L)X(L)U(L+1)$ and $B = X(1) \cdots X(L)$. Since every element of $V$ has only ones on its diagonal, for all $i,j \in [n]$ this gives
\begin{eqnarray*}
A_{i,j} &=& \sum U(1)_{\rho_0,\rho_1}X(1)_{\rho_1,\rho_2}U(2)_{\rho_2,\rho_3}\cdots X(L)_{\rho_{2L-1},\rho_{2L}}U(L+1)_{\rho_{2L},\rho_{2L+1}}
\end{eqnarray*}
where the sum ranges over all choices of $\rho_i \in [n]$, with $\rho_0=i$ and $\rho_{2L+1} = j$. Since $a+b \geq a,b$ for all $a,b \in S$, it follows that by restricting the choices for the $\rho_i$ we will obtain a partial sum that must be less than or equal to $A_{i,j}$. In particular, we have
\begin{eqnarray*}
A_{i,j} &\geq& \sum U(1)_{\rho_0,\rho_0}X(1)_{\rho_0,\rho_1}U(2)_{\rho_1,\rho_1}\cdots X(L)_{\rho_{L-1},\rho_{L}}U(L+1)_{\rho_{L},\rho_{L}}
\end{eqnarray*}
where the sum ranges over all choices of $\rho_i \in [n]$, with $\rho_0=i$ and $\rho_{L} = j$. Since all diagonal entries of elements of $V$ are equal to $1_S$, this gives
\begin{eqnarray*}
A_{i,j} &\geq& \sum X(1)_{\rho_0,\rho_1}\cdots X(L)_{\rho_{L-1},\rho_{L}},
\end{eqnarray*}
where the sum ranges over all choices of $\rho_i \in [n]$, with $\rho_0=i$ and $\rho_{L} = j$. By the definition of matrix multiplication, the latter is equal to $B_{i,j}$. Thus for all $i,j \in [n]$ we have $A_{i,j} \geq B_{i,j}$, and hence $B \preceq A$.

(ii) It follows immediately from part (i) that the powers are non-decreasing. In particular, if $A \mathcal{J} B$ in $V$ then there exist $P,Q,X,Y \in V$ with $A = PBQ$ and $B=XAY$. Now by part (i) this gives $A \preceq B$ and $B \preceq A$, and hence $A=B$. Recalling that an element $A \in V$ is regular if and only if it is $\mathcal{D}$-related to an idempotent, it follows immediately that $A$ is regular if and only if it is idempotent.
\end{proof}

From now on let $S$ be an interval semiring and define
$$R_n(S) = \{A \in M_n(S): A_{i,i} = 1_S\}.$$
It is easily verified that $R_n(S)$ is a semigroup satisfying the conditions of Lemma~\ref{Jtriv}. Let $Z$ be the element of $R_n(S)$ given by $Z_{i,j} = 1_S$ for all $i$ and $j$. Then it is easy to see that $I_n \preceq A \preceq Z$ for all $A \in R_n(S)$, with $AZ = Z = ZA$. In the case where $S=\mathbb{B}$, it is clear that $R_n(\mathbb{B})$ is isomorphic to the monoid $R_n$ of reflexive binary relations on a set of cardinality $n$.

Let $\rho:=(\rho_0, \ldots, \rho_L)$ be an $L+1$-tuple of elements from $[n]$. We shall say that $\rho$ is a \emph{block chain} of length $L+1$ if $\rho$ has the form:
$$\rho:=(i_0, \ldots,i_0, i_1, \ldots, i_1, \ldots, i_k \ldots, i_k),$$
where $i_0, \ldots, i_k$ are \emph{distinct} elements of $[n]$ and thus, $k \leq n-1$.

\begin{lemma}\label{aperiodic}
Let $S$ be an interval semiring.
\begin{itemize}
\item[(i)] If $A = X(1) \cdots X(L)$ in $R_n(S)$, then for all $i,j \in [n]$ we have
$$A_{i,j} = \sum X(1)_{\rho_0, \rho_1}X(2)_{\rho_1, \rho_2} \cdots X(L)_{\rho_{L-1}, \rho_L},$$
where the sum ranges over all block chains $\rho:=(\rho_0, \ldots, \rho_L)$ with $\rho_0=i$ and $\rho_L=j$.
\item[(ii)] For all $A \in R_n(S)$ and all $N \geq n-1$ we have $A^N = A^{n-1}$.\\
(In particular, $A^{n-1}$ is idempotent and $R_n(S)$ is aperiodic.)
\end{itemize}
\end{lemma}

\begin{proof}
(i) Let $A = X(1) \cdots X(L)$ in $R_n(S)$. Then, by the definition of matrix multiplication,
$$A_{i,j} = \sum X(1)_{\rho_0,\rho_1} X(2)_{\rho_1, \rho_2} \cdots X(L)_{\rho_{L-1}, \rho_{L}},
$$
where the sum ranges over all $L+1$-tuples $\rho:=(\rho_0, \ldots, \rho_L)$, with $\rho_k \in [n]$ and $\rho_0=i, \rho_L=j$. Let $\rho$ be such a tuple, and suppose that $\rho$ is not a block chain. Then for some  $s,t$ with $s+1<t$ we must have $\rho_s \neq \rho_{s+1}$ and $\rho_s=\rho_t$. Consider the tuple $\rho':=(\rho'_0, \ldots, \rho'_L)$ obtained from $\rho$ by replacing each $\rho_k$ with $s<k <t$ by $\rho_s$. Since each diagonal entry is equal to $1_S$ and $1_S \geq a$ for all $a \in S$, it is easy to see that:
$$X(1)_{\rho'_0,\rho'_1} X(2)_{\rho'_1, \rho'_2} \cdots X(L)_{\rho'_{L-1}, \rho'_{L}} \geq X(1)_{\rho_0,\rho_1} X(2)_{\rho_1, \rho_2} \cdots X(L)_{\rho_{L-1}, \rho_{L}}.$$
By repeated application of the above argument, it is clear that
$$X(1)_{\sigma_0,\sigma_1} X(2)_{\sigma_1, \sigma_2} \cdots X(L)_{\sigma_{L-1}, \sigma_{L}} \geq X(1)_{\rho_0,\rho_1} X(2)_{\rho_1, \rho_2} \cdots X(L)_{\rho_{L-1}, \rho_{L}},$$
for some block chain $\sigma$. Since $a \leq b$ in $S$ if and only if $a+b=b$, it follows from the previous observation that taking the sum over all block chains must give the same result as taking the sum over all tuples. Thus
$$A_{i,j} = \sum X(1)_{\rho_0,\rho_1} X(2)_{\rho_1, \rho_2} \cdots X(L)_{\rho_{L-1}, \rho_{L}},
$$
where the sum ranges over all block chains $\rho:=(\rho_0, \ldots, \rho_L)$ with $\rho_0=i$ and $\rho_L=j$.

(ii) Let $A \in R_n(S)$ and $N \in \mathbb{N}$. Then by part (i)
$$(A^N)_{i,j} = \sum A_{\rho_0,\rho_1} A_{\rho_1, \rho_2} \cdots A_{\rho_{N-1}, \rho_{N}},
$$
where the sum ranges over all $N+1$-tuples of the form
$$\rho:=(i, \ldots,i, i_1, \ldots, i_1, \ldots, i_k \ldots, i_k, j\ldots, j),$$ where $i, i_1, \ldots, i_k, j$ are \emph{distinct} elements of $[n]$. Moreover, for such an $N+1$-tuple $\rho$, the fact that the diagonal entries of $A$ are all equal to $1_S$ means that the corresponding term of the summation is equal to
$$A_{i,i_1}A_{i_1,i_2} \cdots A_{i_{k-1}, i_k} A_{i_k,j}.$$
Thus for each $N\geq n-1$ we see that every term occurring in the summation above also occurs as a term in the corresponding summation for $A^{n-1}$, and hence $A^N \preceq A^{n-1}$. On the other hand, by Lemma \ref{Jtriv}, we know that $A^{n-1} \preceq A^{N}$ for all $N \geq n-1$. Thus we may conclude that $A^{n-1} = A^N$ for all $N \geq n-1$. In particular,
$$A^{n-1}A^{n-1} = A^{2n-2} = A^{n-1}.$$
(Recall that a semigroup $V$ is aperiodic if for every $a \in V$ there exists a positive integer $m$ such that $a^{m+1} = a^m$.)
\end{proof}

We note that in the case where $R_n(S)$ is finite, the fact that $R_n(S)$ is aperiodic follows directly from Lemma~\ref{Jtriv}, since every finite $\mathcal{H}$-trivial semigroup is aperiodic. For infinite semigroups, $\mathcal{J}$-triviality is not sufficient to deduce aperiodicity (for example, the semigroup of natural numbers under addition is an infinite $\mathcal{J}$-trivial semigroup which is clearly not aperiodic).

\begin{theorem}
\label{thm_reflexive}
Let $S$ be an interval semiring. The identity $w=v$ over alphabet $\Sigma$ is satisfied in $R_n(S)$ if and only if $w$ and $v$ have the same scattered subwords of length at most $n-1$.
\end{theorem}

\begin{proof}
Noting that $U_n(S) \subseteq R_n(S)$, it suffices to show that if $w$ and $v$ have the same scattered subwords of length at most $n-1$, then $w=v$ holds in $R_n(S)$.

Let $\phi:\Sigma^+ \rightarrow R_n(S)$ be a morphism and let $w=w_1 \cdots w_q \in \Sigma^+$. By Lemma \ref{aperiodic} for each $i,j \in [n]$ we have
\begin{eqnarray*}
\phi(w)_{i,j} &=&  (\phi(w_1)\cdots
\phi(w_q))_{i,j}\\
&=&  \sum \phi(w_{1})_{\rho_0, \rho_1} \cdots
\phi(w_{q})_{\rho_{q-1}, \rho_{q}},
\end{eqnarray*}
where the sum ranges over all block chains $\rho$ of total length $q+1$, with first entry $i$ and last entry $j$. To each choice of $t=(t_0,t_1,\ldots,
t_p, t_{p+1})$ with $0=t_0<t_1 < \cdots < t_p<t_{p+1}=q+1$ and $p \leq n-1$ we may associate the set $B_t^{i,j}$ of all block chains of the form:
$$(\underbrace{i_0, \ldots, i_0}_{t_1-t_0}, \underbrace{i_1, \ldots, i_1}_{t_2-t_1}, \underbrace{i_2, \ldots, i_2}_{t_3-t_2} \ldots, \underbrace{i_{p-1}, \ldots, i_{p-1}}_{t_p-t_{p-1}} \underbrace{i_p, \ldots, i_p}_{t_{p+1}-t_p})$$
with $i_0=i$, $i_p=j$. It is easy to see that the set of all block chains of total length $q+1$ with first entry $i$ and last entry $j$ is the disjoint union of the sets $B_t^{i,j}$. Thus the summation above can be viewed as summing over all block chains in $B_t^{i,j}$ for all choices $0=t_0<t_1 < \cdots < t_p<t_{p+1}=q+1$.

Fix $t$ and consider the term of the summation corresponding to the block chain
$$(\underbrace{i_0, \ldots, i_0}_{t_1-t_0}, \underbrace{i_1, \ldots, i_1}_{t_2-t_1}, \underbrace{i_2, \ldots, i_2}_{t_3-t_2} \ldots, \underbrace{i_{p-1}, \ldots, i_{p-1}}_{t_p-t_{p-1}} \underbrace{i_p, \ldots, i_p}_{t_{p+1}-t_p}).$$
The fact that all diagonal entries are equal to $1_S$ means that the corresponding term is equal to
$$\phi(w_{t_1})_{i_0,i_1}\phi(w_{t_2})_{i_1,i_2} \cdots \phi(w_{t_p})_{i_{p-1}, i_p}.$$
It is then clear that the above expression depends only upon the choice of scattered subword $u=w_{t_1}\cdots w_{t_p}$ of $w$ of length $p \leq n-1$, and the intermediate vertices $i_1, \ldots, i_{p-1}$. Since addition in $S$ is idempotent, we may therefore conclude that
$$\phi(w)_{i,j} =  \sum \phi(u_1)_{i_0,i_1}\phi(u_2)_{i_1,i_2} \cdots \phi(u_p)_{i_{p-1},i_p},
$$
where the sum ranges over all scattered subwords $u$ of 
$w$  of length at most $n-1$, and over all choices of distinct $i_0, \ldots,i_p \in [n]$ with $i_0=i$ and $i_p=j$. It then follows that if $w$ and $v$ contain the same scattered subwords of length at most $n-1$ then $\phi(w)=\phi(v)$.
\end{proof}

\section{Catalan monoids and gossip}
\label{sec_catalan}
The Catalan monoid $C_n$ \cite{Solomon} is the monoid given by the presentation with generators $e_1,\ldots, e_{n-1}$ and relations 
\begin{equation}
\label{rels}
e_ie_i=e_i, \;\;\;\;\;e_ie_j=e_je_i \;\;\;\;\; e_ie_{i+1}e_i=e_{i+1}e_{i}e_{i+1} =e_i e_{i+1}
\end{equation}
for all appropriate $i,j$ with $|i-j|>1$. The name comes from the fact that $|C_n| = \frac{1}{n+1} \binom{2n}{n}$ is the $n$th Catalan number.

Say that a matrix $A \in M_n(\mathbb{B})$ is \emph{convex} if:
\begin{itemize}
\item[(1)] $A_{i,l}=A_{i,r}=1$ with $l \leq r$ implies $A_{i,k}=1$ for all $l \leq k \leq r$,
\item[(2)] $ A_{u,j}=A_{d,j}=1$ with $u \leq d$ implies $A_{k,j}=1$ for all $u \leq k \leq d$, and
\item[(3)] $A_{i,i}=1$ for all $i$.
\end{itemize}
By \cite[Proposition 3]{MS12} the set ${\rm Conv}_n$ of all convex Boolean matrices is a submonoid of $R_n$. Let $C_n^U = {\rm Conv}_n \cap U_n$ denote the monoid of all convex upper unitriangular matrices, and for $1 \leq i\leq n-1$ let $D(i)\in C_n^U$ be the matrix with $1$'s on the diagonal and a single off-diagonal $1$ in position $(i,i+1)$.
\begin{lemma}
\label{convex}
The matrices $D(1), \ldots, D(n-1)$ generate the monoid $C_n^U$ of all convex upper unitriangular Boolean matrices. Moreover, $C_n^U \cong C_n$.
\end{lemma}

\begin{proof}
Since each $D(i)$ is a convex upper unitriangular matrix, these matrices clearly generate a submonoid of  $C_n^U$. Suppose then that $A \in C_n^U$. Let $m_i = {\rm max}\{j: A_{i,j}=1\}$. Since $A_{i,i}=1$ we note that $m_i \geq i$. Convexity of $A$ yields that $m_i \leq m_j$ whenever $i \leq j$. Define
$$F(i) = \begin{cases}
I_n & \mbox{ if } m_i=i\\
D(i) \cdots D(m_i-1) & \mbox{ if } m_i>i\\
\end{cases}$$
It is straightforward to verify that if $F(i)_{i,j}=1$ if and only if $i \leq j \leq m_i$, and similarly for all $k>i$, we have $F(k)_{i,j}=1$ if and only if $i=j$. Thus the $(i,j)$th coordinate of  $B:=F(n-1) \cdots F(i)$ is non-zero if and only if $i \leq j \leq m_i$. Let $M=BF(i-1)\cdots F(1)$. We claim $A=M$. Since $M \succeq B$, it is clear from the observations above that $M_{i,j}\geq B_{i,j}=1$ for all $i \leq j \leq m_i$. Since $M$ is clearly upper triangular, it remains to show that $M_{i,j}=0$ for all $j > m_i$. To see this, notice that the right action of $D(k)$ on any Boolean matrix $X$ results in the matrix obtained from $X$ by taking the Boolean sum of columns $k$ and $k+1$. By definition, all factors $D(k)$ occurring in $F(j)$ satisfy $j \leq k \leq m_j-1$. For $j<i$ the only factors $D(k)$ occurring in $F(j)$ therefore satisfy $j \leq k \leq m_j-1 \leq m_i-1$. This means that $M$ is obtained from the matrix $B$ by the right action of some collection of matrices $D(k)$ with $k\leq m_i-1$, and hence columns $j >m_i$ of $M$ and $B$ agree.

It is straightforward to verify that the matrices $D(i)$ satisfy the relations \eqref{rels}. Since the elements of $C_n^U$ are in one to one correspondence with the Dyck paths from $(0,0)$ to $(n,n)$, we see that $|C_n^U|=|C_n|$, and so these two monoids must be isomorphic. 
\end{proof}

Let $E(i)$ denote the product $D(i)D(i)^T \in {\rm Conv}_n$. The double Catalan monoid $DC_n$ of Mazorchuk and Steinberg \cite{MS12} is the submonoid of ${\rm Conv}_n$ generated by the matrices $E_1, \ldots, E_{n-1}$. Define $\mathcal{U}: DC_n \rightarrow C_n^U$ to be the map sending a matrix to its \emph{upper profile}, namely $\mathcal{U}(A)_{i,j} = A_{i,j}$ if $i \leq j$ and $\mathcal{U}(A)_{i,j} =0$ otherwise.
\begin{lemma}
\label{surjection}
The map $\mathcal{U}: DC_n \rightarrow C_n^U$ is a surjective monoid homomorphism.
\end{lemma}

\begin{proof}
Let $A, B \in DC_n$. By definition $\mathcal{U}(AB)_{i,j} = (AB)_{i,j}$ if $i \leq j$ and $0$ otherwise. Thus the non-zero entries occur in positions $i \leq j$ for which there exists $k$ with $A_{i,k} = B_{k,j}=1$. Note that if there exists such a $k$ with $k<i$, then by the convexity of $B$ we must have $A_{i,i}=B_{i,j}=1$, whilst if there exists such a $k$ with $k>j$, then by the convexity of $A$ we must have $A_{i,j}=B_{j,j}=1$. The non-zero entries of  $\mathcal{U}(AB)$ therefore occur in positions $i,j$ for which there exists $i \leq k \leq j$ with $A_{i,k} = B_{k,j}=1$, and it is easy to see that these coincide with the non-zero entries of  $\mathcal{U}(A) \mathcal{U}(B)$.

Now let $A \in DC_n$. By definition we may write $A = E_{i_1} \cdots E_{i_m}$ for some $1 \leq i_1, \ldots, i_m \leq n$. Applying the morphism $\mathcal{U}$ then yields
$$\mathcal{U}(A) = \mathcal{U}(E_{i_1}) \cdots \mathcal{U}(E_{i_m}) = D_{i_1} \cdots D_{i_m},$$
and the result follows from Lemma \ref{convex}.
\end{proof}

Now let $D(i,j)$ denote the $n \times n$ Boolean matrix with $1$'s on the diagonal and a single off-diagonal $1$ in position $(i,j)$, and let $E(i,j) = D(i,j)D(j,i)$. The \emph{gossip monoid} \cite{BDF15, FJK16} is the submonoid of $M_n(\mathbb{B})$ generated by the set $\{E(i,j): 1 \leq i < j \leq n\}$. The \emph{one directional gossip monoid} $\overline{G}_n$ is the submonoid of $M_n(\mathbb{B})$ generated by the set $\{D(i,j): 1 \leq i \neq j \leq n\}$. It is clear from the definition that $G_n$ is a submonoid of $\overline{G}_n$. Moreover, since $E(i) = E(i,i+1)$ we see that the double Catalan monoid is a submonoid of $G_n$. The names `one-directional gossip monoid' and `gossip monoid' come from the following interpretation of the matrices $D(i,j)$ and $E(i,j)$. Consider a group of $n$ people, each with a unique piece of information or `gossip' they would like to spread. It is clear that we can record the state of knowledge amongst the $n$ people at any given time by means of a Boolean matrix, putting a $1$ in the $(i,j)$th position if and only if person $j$ has learned the piece of gossip originally known only to person $i$. The right action of the matrix $D(i,j)$ on $M_n(\mathbb{B})$ then corresponds to a one-way communication from person $i$ to person $j$, in which person $i$ recounts to person $j$ all of the gossip that they know. The right action of the matrix $E(i,j)$ on $M_n(\mathbb{B})$ corresponds to a two-way communication between person $i$ and person $j$, at the end of which both parties have learned the sum total of gossip known to either $i$ or $j$. The double Catalan monoid can therefore be thought of as an algebraic model of gossip in a network in which person $i$ can communicate only with their nearest neighbours, $i-1$ and $i+1$.

\begin{proposition}
\label{gossip}
Let $n \in \mathbb{N}$. The gossip monoid $G_n$, the one-directional gossip monoid $\bar{G}_n$, and the double Catalan monoid $DC_n$, all satisfy the same set of identities as the reflexive monoid $R_n$.
\end{proposition}

\begin{proof}
It is clear from the above definitions that $DC_n \subseteq G_n \subseteq \overline{G_n} \subseteq R_n$. Thus  ${\rm Id}(DC_n) \supseteq {\rm Id}(G_n) \supseteq {\rm Id}(\overline{G_n}) \supseteq \rm{Id}(R_n)$. By  Lemma \ref{surjection}, there is a surjective monoid homomorphism from $\mathcal{U}: DC_n \rightarrow C_n^U$, from which it follows that ${\rm Id}(C_n^U) \supseteq {\rm Id}(DC_n)$. The result then follows from the fact that $R_n$ and $C_n$ satisfy the same identities \cite{V04}.
\end{proof}

Now let $S$ be an interval semiring and for each $s \in S$ define: $D(i,j;s)$ to be the matrix with $1$'s on the diagonal and a single off-diagonal entry $s$ in position $(i, j)$;
and  $E(i,j; s)=D(i,j;s)D(j,i; s)$. Then we may define monoids:
\begin{eqnarray*}
C_n^U(S)&:=& \langle D(i,i+1; s): 1 \leq i \leq n-1, s \in S \rangle\\
DC_n(S)&:=& \langle E(i,i+1; s): 1 \leq i \leq n-1, s \in S \rangle\\
G_n(S)&:=& \langle E(i,j; s): 1 \leq i,j \leq n-1, s \in S \rangle\\
\overline{G_n}(S)&:=& \langle D(i,j; s): 1 \leq i,j \leq n-1, s \in S \rangle
\end{eqnarray*}
 Since $S$ is an interval semiring, we note that each is a submonoid of $R_n(S)$.

\begin{proposition}
\label{fatalan}
Let $S$ be an interval semiring.  The monoids $C_n^U(S)$, $DC_n(S)$, $G_n(S)$ and $\overline{G_n}(S)$ satisfy the same identities as the monoid $R_n(S)$.
\end{proposition}

\begin{proof}
It is clear from the definitions that 
$$DC_n \subseteq DC_n(S) \subseteq \overline{G_n}(S) \subseteq G_n(S) \subseteq R_n(S)$$
and
$$C_n^U \subseteq C_n^U(S) \subseteq  R_n(S).$$
Thus by Proposition \ref{fatalan} and Theorem \ref{thm_reflexive} we deduce that each of these monoids satisfies the same set of identities. 
\end{proof}

In the case where $S$ is the subsemiring $[ 0, +\infty]$ of the (min-plus) tropical semiring, it is straightforward to verify that the monoid $G_n(S)$ is precisely the lossy gossip monoid $\mathcal{G}_n$ of \cite{BDF15}.

\begin{corollary}
The lossy gossip monoid is finitely based for $n \leq 4$ and non-finitely based otherwise.
\end{corollary}

\section*{Acknowledgements} The authors thank Professor Volkov for suggesting the topic of this article.

\bibliography{refs}
\bibliographystyle{plain}
\end{document}